\documentclass[11pt,reqno]{article}

\usepackage{amsmath,amsfonts,amssymb,latexsym,amsthm,dsfont,graphicx}

\usepackage{ifpdf}
\ifpdf
\usepackage[pdftex]{hyperref}
\else
\usepackage[dvips]{hyperref}
\fi
\hypersetup{a4paper,colorlinks}
\hypersetup{bookmarksnumbered,plainpages=false,hypertexnames=false}
\hypersetup{pagecolor=blue,linkcolor=blue,citecolor=blue,urlcolor=blue}

\usepackage[a4paper,portrait]{geometry}
\title{Asymptotic analysis and diffusion limit \\ %
  of the Persistent Turning Walker Model}

\author{%
  Patrick~\textsc{Cattiaux}, %
  Djalil~\textsc{Chafa\"\i}, %
  S\'ebastien~\textsc{Motsch}}

\date{Preprint -- October 2008 -- Revised September 2009}

\newtheorem{theorem}[equation]{Theorem}

\newtheorem{lemma}[equation]{Lemma}
\newtheorem{corollary}[equation]{Corollary}

\theoremstyle{definition}
\newtheorem{remark}[equation]{Remark}

\numberwithin{equation}{section}
\newcommand{\ote}{\underline{\theta}}
\newcommand{\cF}{\mathcal{F}}
\newcommand{\cW}{\mathcal{W}}

\newcommand{\dR}{\mathbb{R}}

\newcommand{\dZ}{\mathbb{Z}}

\newcommand{\dP}{\mathbb{P}}
\newcommand{\dE}{\mathbb{E}}

\newcommand{\dD}{\mathbb{D}}
\newcommand{\dL}{\mathbb{L}}

\newcommand{\Var}{\mathrm{Var}}
 
\newcommand{\ABS}[1]{{{\left| #1 \right|}}} % |1|
\newcommand{\DP}[1]{{{\left<#1\right>}}} % <1>
\newcommand{\NRM}[1]{{{\left\| #1\right\|}}} % ||1||
\newcommand{\PAR}[1]{{{\left(#1\right)}}} % (1)
\newcommand{\SBRA}[1]{{{\left[#1\right]}}} % [1]
\begin{document}

\maketitle

\begin{abstract}
  The Persistent Turning Walker Model (PTWM) was introduced by Gautrais et al
  in Mathematical Biology for the modelling of fish motion. % \cite{gautetal}.
  It involves a nonlinear pathwise functional of a non-elliptic hypo-elliptic
  diffusion. This diffusion solves a kinetic Fokker-Planck equation based on
  an Ornstein-Uhlenbeck Gaussian process. The long time ``diffusive'' behavior
  of this model was recently studied by Degond \& Motsch % in \cite{DegSeb}
  using partial differential equations techniques. This model is however
  intrinsically probabilistic. In the present paper, we show how the long time
  diffusive behavior of this model can be essentially recovered and extended
  by using appropriate tools from stochastic analysis. The approach can be
  adapted to many other kinetic ``probabilistic'' models. 
  % Beyond the mathematical results, the aim of this short paper is also to
  % contribute to the diffusion of stochastic techniques in the domain of
  % partial differential equations. Also, the text aims to be very accessible
  % for non probabilists.
\end{abstract}

\bigskip

{\footnotesize %

\medskip

\noindent\textbf{AMS-MSC.} 82C31; 35H10; 60J60; 60F17; 92B99; 92D50; 34F05.

}

\section{Introduction}\label{se:intro}

Different types of models are used in Biology to describe individual
displacement. For instance, correlated/reinforced random walks are used for
the modelling of ant, see e.g. \cite{casellas2008icd,theraulaz2002spa}, and
cockroaches, see e.g. \cite{jeanson2003mam} and \cite{codling:rwm} for a
review. On the other hand, a lot of natural phenomena can be described by
kinetic equations and their stochastic counterpart, stochastic differential
equations. The long time behavior of such models is particularly relevant
since it captures some ``stationary'' evolution. Recently, a new model, called
the Persistent Turning Walker model (PTWM for short), involving a kinetic
equation, has been introduced to describe the motion of fish
\cite{gautetal,DegSeb}. The natural long time behavior of this model is
``diffusive'' and leads asymptotically to a Brownian Motion.

The diffusive behavior of the PTWM has been obtained in \cite{DegSeb} using
partial differential equations techniques. In the present work, we show how to
recover this result by using appropriate tools from stochastic processes
theory. First, we indicate how the diffusive behavior arises naturally as a
consequence of the Central Limit Theorem (in fact an Invariance Principle). As
expected, the asymptotic process is a Brownian Motion in space. As a
corollary, we recover the result of \cite{DegSeb} which appears as a special
case where the variance of the Brownian Motion can be explicitly computed. We
finally extend our main result to more general initial conditions. We
emphasize that the method used in the present paper is not restricted to the
original PTWM. In particular, the hypotheses for the convergence enables to
use more general kinetic models than the original PTWM.

The present paper is organized as follows: in Section \ref{se:prel}, we recall
the PTWM and its main properties, and we give the main results. Section
\ref{se:proofs} is dedicated to the proofs.

% Trend to such an ``equilibrium'' can be studied using various methods.
% Fluctuations ``at equilibrium'' are expected to model a ``diffusive''
% behavior. The present work focuses on a specific model called the
% ``Persistent Turning Walker Model'' (PTWM for short) used for the modelling
% of fish motion in Biology. We refer to \cite{gautetal,DegSeb} for a detailed
% discussion of the pertinence of the model. The

\section{Main results}\label{se:prel}

In the PTWM, the motion is described using three variables: \emph{position}
$x\in \dR^2$, \emph{velocity angle} $\theta \in \dR$, and \emph{curvature}
$\kappa \in \dR$. For some fixed real constant $\alpha$, the probability
distribution $p(t,x,\theta,\kappa)$ of finding particles at time $t$ in a
small neighborhood of $(x,\theta,\kappa)$ is given by a forward
Chapman-Kolmogorov equation
\begin{equation}\label{eq:Kforward}
  \partial_t p %
  + \tau.\nabla_x p %
  + \kappa \partial_\theta p %
  - \partial_\kappa(\kappa p) %
  - \alpha^2 %
  \, \partial^2_{\kappa^2} p \, %
  = \, 0
\end{equation}
with initial value $p_0$, where
$$
\tau(\theta)=(\cos \theta, \sin \theta)=e^{\sqrt{-1}\,\theta}.
$$
The stochastic transcription of \eqref{eq:Kforward} is given by the stochastic
differential system $(t\geq0)$
\begin{equation}\label{eq:sdefish}
\begin{cases}
  dx_t = \tau(\theta_t) \, dt \\
  d\theta_t  = \kappa_t \, dt \\
  d\kappa_t  = - \kappa_t \, dt \, + \, \sqrt{2} \alpha \, dB_t
\end{cases}
\end{equation}
where $(B_t)_{t\geq0}$ is a standard Brownian Motion on $\dR^2$. The
probability density function $p(t,x,\theta,\kappa)$ of
$(x_t,\theta_t,\kappa_t)$ with a given initial law $p_0 \, dx \, d\theta \,
d\kappa$ is then solution of \eqref{eq:Kforward}. Also, \eqref{eq:Kforward} is
in fact a kinetic Fokker-Planck equation. Note that $(\kappa_t)_{t\geq0}$ is
an Ornstein-Uhlenbeck Gaussian process. The formula
$$
\theta_t=\theta_0+\int_0^t\!\kappa_s\,ds
$$
expresses $(\theta_t)_{t\geq0}$ as a pathwise linear functional of
$(\kappa_t)_{t\geq0}$. In particular the process $(\theta_t)_{t\geq0}$ is
Gaussian and is thus fully characterized by its initial value, together with
its time covariance and mean which can be easily computed from the ones of
$(\kappa_t)_{t\geq0}$ conditional on $\theta_0$ and $\kappa_0$. The process
$(\theta_t)_{t\geq0}$ is not Markov. However, the pair
$(\theta_t,\kappa_t)_{t\geq0}$ is a Markov Gaussian diffusion process and can
be considered as the solution of a degenerate stochastic differential
equation, namely the last two equations of the system \eqref{eq:sdefish}.
Additionally, the process $(x_t)_{t\geq0}$ is an ``additive functional'' of
$(\theta_t,\kappa_t)_{t\geq0}$ since
\begin{equation}\label{eq:addfish}
  x_t %
   \ = \ x_0 + \int_0^t\!\tau(\theta_s) \, ds  %
   \ = \ x_0 +\int_0^t\!\tau\PAR{\theta_0+\int_0^s\!\kappa_u\,du}\,ds.
\end{equation}
%In the sequel, we take $x_0=0$. 
Note that $x_t$ is a nonlinear function of $(\theta_s)_{0\leq s\leq t}$ due to
the nonlinear nature of $\tau$, and thus $(x_t)_{t\geq0}$ is not Gaussian. The
invariant measures of the process $(\theta_t,\kappa_t)_{t\geq0}$ are multiples
of the tensor product of the Lebesgue measure on $\dR$ with the Gaussian law
of mean zero and variance $\alpha^2$. These measures cannot be normalized into
probability laws. Since $\tau$ is $2\pi$-periodic, the process
$(\theta_t)_{t\geq0}$ acts in the definition of $x_t$ only modulo $2\pi$, and
one may replace $\theta$ by $\ote \in S^1:=\dR/2\pi \dZ$. The Markov diffusion
process
$$
{(y_t)}_{t\geq0}={(\ote_t,\kappa_t)}_{t\geq0}
$$
has state space $S^1\times\dR$ and admits a unique invariant law $\mu$ which
is the tensor product of the uniform law on $S^1$ with the Gaussian law of
mean zero and variance $\alpha^2$, namely
$$
d\mu(\ote,\kappa) %
= \frac{1}{\sqrt{2\pi\alpha^2}}\mathds{1}_{S^1}(\ote)%
\exp\PAR{-\frac{\kappa^2}{2\alpha^2}}%
 d\ote d\kappa.
$$
Note that ${(y_t)}_{t\geq0}$ is ergodic but is not reversible (this is simply
due to the fact that the dynamics on observables depending only on $\ote$ is
not reversible). The famous Birkhoff-von Neumann Ergodic Theorem
\cite{MR797411,MR1725357,JS,Kut,MR997938} states that for every
$\mu$-integrable function $f:S^1\times\dR\to\dR$ and any initial law $\nu$
(i.e.\ the law of $y_0$), we have,
\begin{equation}\label{eq:erg}
  \mathbb{P}\PAR{
    \lim_{t\to\infty}%
    \PAR{\frac{1}{t}\int_0^t\!f(y_s)\,ds-\int_{S^1\times\dR}\!f\,d\mu}=0}=1.
\end{equation}
Beyond this Law of Large Numbers describing for instance the limit of the
functional \eqref{eq:addfish}, one can ask for the asymptotic fluctuations,
namely the long time behavior as $t\to\infty$ of
\begin{equation}\label{eq:CLT}
  \sigma_t\PAR{\frac{1}{t}\int_0^t\!f(y_s)\,ds-\int_{S^1\times\dR}\!f\,d\mu}
\end{equation}
where $\sigma_t$ is some renormalizing constant such that $\sigma_t\to\infty$
as $t\to\infty$. By analogy with the Central Limit Theorem (CLT for short) for
reversible diffusion processes (see e.g. \cite{MR2238823,Kut}), we may expect,
when $f$ is ``good enough'' and when $\sigma_t=\sqrt{t}$, a convergence in
distribution of \eqref{eq:CLT} as $t\to\infty$ to some Gaussian distribution
with variance depending on $f$ and on the infinitesimal dynamics of
${(y_t)}_{t\geq0}$. This is the aim of Theorem \ref{th:main} below, which goes
actually further by stating a so called \emph{Invariance Principle}, in other
words a CLT for the whole process and not only for a fixed single time.

\begin{theorem}[Invariance Principle at equilibrium]\label{th:main}
  Assume that $y_0=(\ote_0,\kappa_0)$ is distributed according to the
  equilibrium law $\mu$. Then for any $C^\infty$ bounded
  $f:S^1\times\dR\to\dR$ with zero $\mu$-mean, the law of the process
  $$
  {(z_t^\varepsilon)}_{t\geq0}:=
  \PAR{\varepsilon\int_0^{t/\varepsilon^2}\!\!\!\!f(y_s)\,ds,\, y_{t/\varepsilon^2}}_{t \geq0}
  $$
  converges as $\varepsilon \to 0$ to $\mathcal W^f \otimes
  \mu^{\otimes\infty}$ where $\mathcal W^f$ is the law of a Brownian Motion
  with variance
  $$
  V_f = -\int\!gLg\,d\mu %
  =2\alpha^2\,\int\!|\partial_\kappa g|^2\,d\mu
  $$
  where $L=\alpha^2 \partial_\kappa^2 - \kappa \partial_\kappa + \kappa
  \, \partial_\theta$ acts on $2\pi$-periodic functions in $\theta$, and
  $g:S^1\times\dR\to\dR$ is
  $$
  g(y)= -\dE\left(\int_0^{\infty}\!f(y_s)\,ds\,\biggr\vert\,y_0=y\right).
  $$
  In other words, for any fixed integer $k\geq1$, any fixed times $0\leq
  t_1<\cdots< t_k$, and any bounded continuous function $F:(\dR\times
  S^1\times\dR)^k\to\dR$, we have
  $$
  \lim_{\varepsilon\to0}
  \dE\SBRA{F(z_{t_1}^\varepsilon,\ldots,z_{t_k}^\varepsilon)}
  =\dE\SBRA{F((W^f_{t_1},Y_1),\ldots,(W^f_{t_k},Y_k))}
  $$
  where $Y_1,\ldots,Y_k$ are independent and equally distributed random
  variables of law $\mu$ and where ${(W_t^f)}_{t\geq0}$ is a Brownian Motion
  with law $\cW^f$, independent of $Y_1,\ldots,Y_k$.
\end{theorem}

Theorem \ref{th:main} encloses some decorrelation information as $\varepsilon$
goes to $0$ since the limiting law is a tensor product (just take for $F$ a
tensor product function). Such a convergence in law at the level of the
processes expresses a so called Invariance Principle. Here the Invariant
Principle is \emph{at equilibrium} since $y_0$ follows the law $\mu$.
% Note that $\tau$ is replaced by a more general $2\pi$-periodic function $f$.
The proof of Theorem \ref{th:main} is probabilistic, and relies on the fact
that $g$ solves the Poisson\footnote{It is amusing to remark that ``poisson''
  means ``fish'' in French\ldots.} equation $Lg=f$. Note that neither the
reversible nor the sectorial assumptions of \cite{MR2238823} are satisfied
here.

Theorem \ref{th:main} remains valid when $f$ is complex valued (this needs the
computation of the asymptotic covariance of the real and the imaginary part of
$f$). The hypothesis on $f$ enables to go beyond the original framework of
\cite{DegSeb}. For instance, we could add the following rule in the model:
{\it the speed of the fish decreases as the curvature increases}.
Mathematically, this is roughly translated as:
\begin{equation}
  \label{eq:f_curvature}
  f(y)=f(\theta,\kappa) = c(|\kappa|) (\cos \theta,\,\sin \theta)
\end{equation}
where $s\mapsto c(s)$ is a regular enough decreasing function, see Figure
\ref{fig:example_PTWM} for a simulation.

\begin{figure}[ht!]
  \centering
  \includegraphics[scale=1]{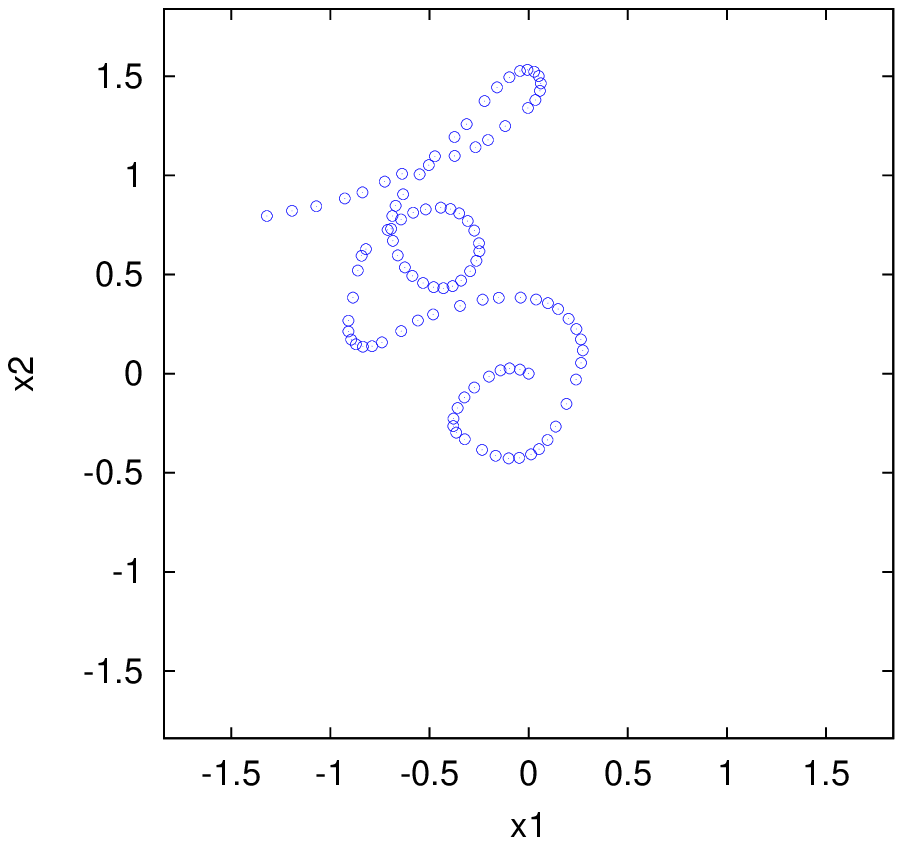}
  \caption{An example of the trajectory $t\mapsto x_t=(x_t^1,x_t^2)$ of the
    PTWM where the speed of the fish decreases with higher curvature (eq.
    \ref{eq:f_curvature}). Here $\alpha=1$ and $c(|\kappa|)=1/(1+2|\kappa|)$.
    The simulation is run during 10 time units, we plot a point each $.1$ time
    unit.}
  \label{fig:example_PTWM}
\end{figure}

The following corollary is obtained from Theorem \ref{th:main} by taking
roughly $f=\tau$ and by computing $V_\tau$ explicitly. In contrast with the
function $f$ in Theorem \ref{th:main}, the function $\tau$ is complex valued.
Also, an additional argument is in fact used in the proof of Corollary
\ref{co:dm} to compute the asymptotic covariance of the real and imaginary
parts of the additive functional based on $\tau$ (note that this seems to be
missing in \cite{DegSeb}).

\begin{corollary}[Invariance Principle for PTWM at equilibrium]\label{co:dm}
  Assume that the initial value $y_0=(\ote_0,\kappa_0)$ is distributed
  according to the equilibrium $\mu$. Then the law of the process
  \begin{equation}\label{eq:proc}
  \PAR{\varepsilon\int_0^{t/\varepsilon^2}\!\!\!\!\tau(\theta_s)\,ds,\,y_{t/\varepsilon^2}}_{t\geq0}
  \end{equation}
  converges as $\varepsilon \to 0$ to $\cW^\tau \otimes \mu^{\otimes\infty}$
  where $\cW^\tau$ is the law of a $2$-dimensional Brownian Motion with
  covariance matrix $\dD I_2$ where
  $$
  \dD = \int_0^{\infty}\!e^{- \, \alpha^2 (s-1+e^{-s})}\,ds.
  $$
\end{corollary}

\begin{figure}[ht!]
  \centering
  \includegraphics[scale=0.5]{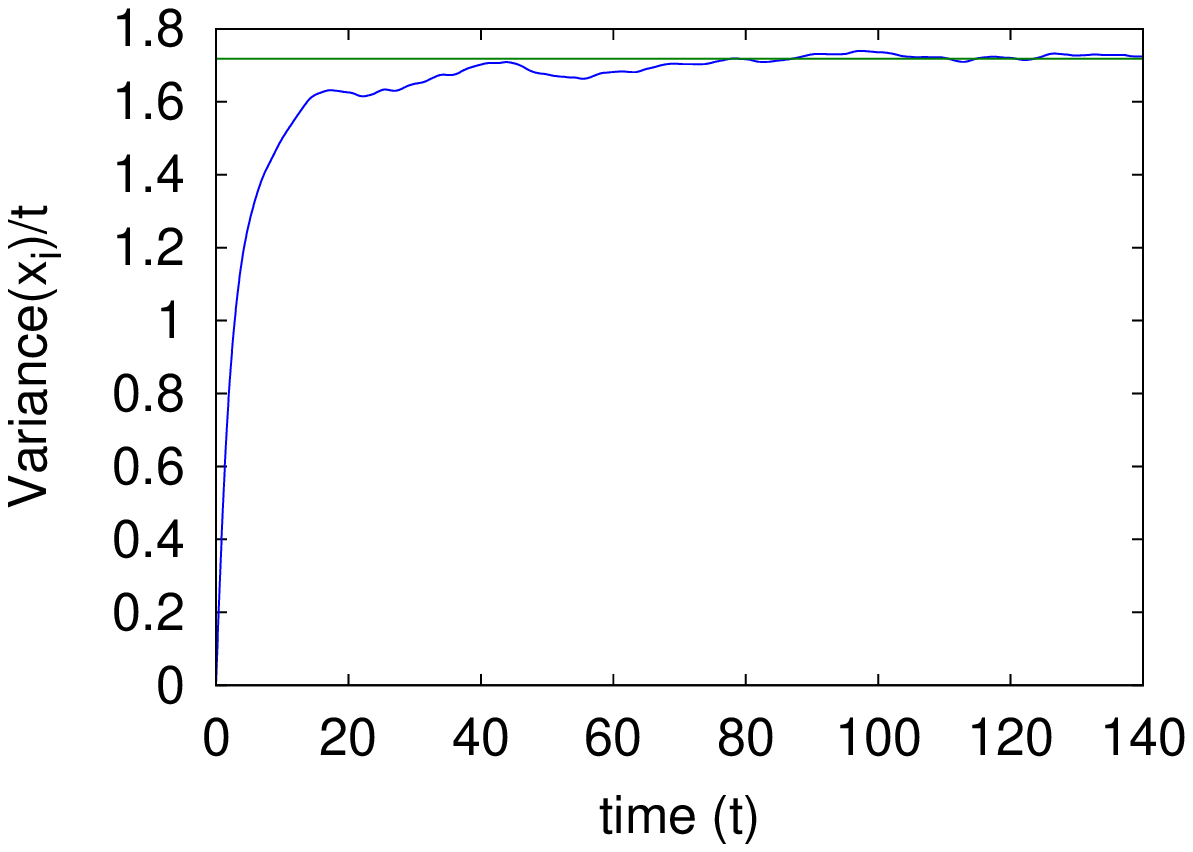}
  \caption{Convergence of $t^{-1}\mathrm{Var}(x_t)$ to the constant $\dD$.
    Here $\alpha=1$.}
  \label{fig:var_div_t}
\end{figure}

It can be shown that the constant $\dD$ which appears in Corollary 
\ref{co:dm} satisfies to
$$
\dD=\lim_{t\to\infty} \frac{1}{t}\Var(x_t^1) %
=\lim_{t\to\infty}\frac{1}{t}\Var(x_t^2) %
\quad\text{where}\quad %
(x_t^1,x_t^2)=x_t=\int_0^{t}\!\!\!\tau(\theta_s)\,ds
$$
see e.g. Figure \ref{fig:var_div_t}. Corollary \ref{co:dm} complements a
result of Degond \& Motsch \cite[Theorem 2.2]{DegSeb} which states -- in their
notations -- that the probability density function
$$
p^\varepsilon(t,x,\ote,\kappa) = \frac{1}{\varepsilon^2}
\, p \,
\left(\frac{t}{\varepsilon^2},\frac{x}{\varepsilon},\ote,\kappa\right) 
$$
converges as $\varepsilon \to 0$ to the probability density
$$
\frac{1}{\sqrt{2\pi}}\,n^0(t,x) \, M(\kappa)
$$ 
where $M$ is the Gaussian law with zero mean and variance $\alpha^2$, and
$n^0$ solves the equation 
$$
\partial_t n^0  - \frac{1}{2}\dD \Delta_x n^0  = 0 
$$ 
where $\dD$ is as in Corollary \ref{co:dm}. Convergence holds in a weak sense
in some well chosen Banach space, depending on the initial distribution. The
meaning of $p^\varepsilon$ is clear from the stochastic point of view: it is
the probability density function of the distribution of the rescaled process
(recall that $x$ is two-dimensional)
$$
\PAR{\varepsilon x_{t/\varepsilon^2},\,y_{t/\varepsilon^2}}_{t\geq0} %
= \PAR{\varepsilon x_{t/\varepsilon^2}, %
  \ \ote_{t/\varepsilon^2}, %
  \ \kappa_{t/\varepsilon^2}}_{t\geq0}.
$$
In other words, the main result of \cite{DegSeb} captures the asymptotic
behavior at fixed time of the process \eqref{eq:proc} by stating that for any
$t$, and as $\varepsilon \to 0$, the law of this process at time $t$ tends to
the law of $(\sqrt{\dD} W_t, \ote,M)$ where $(W_t)_{t\geq0}$, and $(\ote,M)$
are independent, $(W_t)_{t\geq0}$ being a standard Brownian Motion, and
$(\ote,M)$ a random variable following the law $\mu$. This result encompasses
what is expected by biologists i.e.\ a ``diffusive limiting behavior''.

Starting from the equilibrium, Corollary \ref{co:dm} is on one hand stronger
and on the other hand weaker than the result of \cite{DegSeb} mentioned above.
Stronger because it is relative to the full law of the process, not only to
each marginal law at fixed time $t$. In particular it encompasses covariance
type estimates at two different times. Weaker because it is concerned with the
law and not with the density. For the density at time $t$ we recover a weak
convergence, while the one obtained in \cite{DegSeb} using partial
differential equations techniques is of strong nature. We should of course go
further using what is called ``local CLTs'', dealing with densities instead of
laws, but this will require some estimates which are basically the key of the
analytic approach used in \cite{DegSeb}.

Our last result concerns the behavior when the initial law is not the
invariant law $\mu$.

\begin{theorem}[Invariance Principle out of equilibrium]\label{th:out}
  The conclusion of Corollary \ref{co:dm} still holds true when
  $y_0=(\ote_0,\kappa_0)$ is distributed according to some law $\nu$ such
  that $d\nu_{s_0}/d\mu$ belongs to $\dL^q(\mu)$ for some $s_0\geq 0$ and
  $q>1$, where $\nu_{s_0}$ is the law of $y_{s_0}$. This condition is
  fulfilled for instance if $d\nu/d\mu$ belongs to $\dL^q(\mu)$ or if $\nu$
  is compactly supported.
\end{theorem}

\section{Proofs}\label{se:proofs}

The story of CLTs and Invariant Principles for Markov processes is quite
intricate and it is out of reach to give a short account of the literature.
The reader may find however a survey on some aspects in e.g.
\cite{MR2238823,MR2068475,MR1637655,Kut}. Instead we shall exhibit some
peculiarities of our model that make the long time study an (almost) original
problem.
% First of all, as mentioned in Section \ref{se:prel},
The underlying diffusion process $(\theta_t,\kappa_t)_{t\geq0}$ with state
space $\dR^2$ is
% not ergodic. %:
% its invariant measures are multiples of the Lebesgue measure times a
% Gaussian law. This process is also
degenerate in the sense that its infinitesimal generator
\begin{equation}\label{eq:L}
L=\alpha^2 \, \partial^2_{\kappa^2} \, %
- \, \kappa \, \partial_\kappa \, %
+ \, \kappa \, \partial_\theta
\end{equation}
is not elliptic. Fortunately, the operator $\partial_t + L$ is H\"{o}rmander
hypo-elliptic since the ``diffusion'' vector field $X=(0,\alpha^2)$ and the
Lie bracket $[X,Y]=XY-YX$ of $X$ with the ``drift'' vector field
$Y=(\kappa,-\kappa)$ generate the full tangent space at each
$(\theta,\kappa)\in \dR^2$. The drift vector field $Y$ is always required, so
that the generator is ``fully degenerate''. This degeneracy of $L$ has two
annoying consequences:
\begin{enumerate}
\item any invariant measure $\nu$ of $L$ is not symmetric, i.e.\ $\int\!f
  Lg\,d\nu \neq \int\!g Lf\,d\nu$ for some nice functions $f$ and $g$ in
  $\dL^2(\nu)$, for instance only depending on $\theta$.
\item the \emph{carr\'e du champ} of $L$ given here by $\Gamma f =
  \frac{1}{2}L(f^2)-fLf= 2\alpha^2|\partial_\kappa f|^2$ is degenerate, so
  that one cannot expect to use any usual functional inequality such as the
  Poincar\'e inequality (i.e.\ spectral gap, see \cite{MR1845806,royer}) in
  order to study the long time behavior.
\end{enumerate}
This situation is typical for kinetic models. In the more general framework of
homogenization, a slightly more general version of this model has been studied
in \cite{HP}, see also the trilogy \cite{PV1,PV2,PV3} for similar results from
which one can derive the result in \cite{DegSeb}. The main ingredient of the
proof of Theorem \ref{th:main} is the control of the ``rate of convergence''
to equilibrium in the Ergodic Theorem \eqref{eq:erg}, for the process
$(\ote_t,\kappa_t)_{t\geq0}$. % instead of $(\theta_t,\kappa_t)_{t\geq0}$.
We begin with a simple lemma which expresses the propagation of chaos as
$\varepsilon$ goes to $0$.

% , so that we shall start by recalling recent results on this topic before to
% go on with the (usual) strategy to get the Invariance Principle. In all what
% follows convergence in law for a family of processes is understood as
% convergence in law for any finite-dimensional subsequence. It turns out that
% with some extra work it is possible to improve most of the results below by
% using the stronger Skorohod topology on paths \cite{JS}.

\begin{lemma}[Propagation of chaos]\label{le:chaos}
  Assume that $y_0=(\ote_0,\kappa_0)$ is distributed according to the
  equilibrium law $\mu$. Then the law of the process
  $(y^\varepsilon)_{t\geq0}=\PAR{y_{t/\varepsilon^2}}_{t \geq0}$ converges as
  $\varepsilon \to 0$ to $\mu^{\otimes\infty}$. In other words, for any fixed
  integer $k\geq1$, any fixed times $0\leq t_1<\cdots< t_k$, and any bounded
  continuous function $F:(S^1\times\dR)^k\to\dR$, we have
  $$
  \lim_{\varepsilon\to\infty}
  \dE\SBRA{F(y_{t_1}^\varepsilon,\ldots,y_{t_k}^\varepsilon)}
  =\dE\SBRA{F(Y_1,\ldots,Y_k)}
  $$
  where $Y_1,\ldots,Y_k$ are independent and equally distributed random
  variables of law $\mu$.
\end{lemma}

\begin{proof}[Proof of Lemma \ref{le:chaos}]
  Let us denote by $L$ the operator \eqref{eq:L} acting this time on
  $2\pi$-periodic functions in $\theta$, i.e. on functions
  $S^1\times\dR\to\dR$. This operator $L$ generates a non-negative contraction
  semi-group $(P_t)_{t\geq0}=(e^{t L})_{t\geq0}$ in $\mathbb L^2(\mu)$ with
  the stochastic representation $P_t f(y)=\mathbb E[f(y_s)|y_0=y]$ for all
  bounded $f$. We denote by $L^*$ the adjoint of $L$ in $\mathbb L^2(\mu)$
  generating the adjoint semi-group $P_t^*$, i.e.
  $$
  L^*=\alpha^2\partial^2_\kappa-\kappa\partial_\kappa-\kappa\partial_\theta
  $$
  acting again on the same functions. The function
  $H(y)=H(\ote,\kappa)=1+\kappa^2$ satisfies
  \begin{equation}\label{eq:lyap}
    L^* H %
    = -2H+2(\alpha^2+1) %
    \leq -H + 2(\alpha^2+1)\,\mathds{1}_{|\kappa|\leq\sqrt{2\alpha^2+1}}
  \end{equation}
  so $H$ is a Lyapunov function in the sense of \cite[Def. 1.1]{BCG}. Since
  $C=S^1 \times \{|\kappa|\leq\sqrt{2\alpha^2+1}\}$ is compact and the process
  $(y_t)_{t\geq0}$ is regular enough, $C$ is a ``petite set'' in the
  terminology \cite[Def. 1.1]{BCG} of Meyn \& Tweedie. Accordingly we may
  apply \cite[Th. 2.1]{BCG} and conclude that there exists a constant $K_2>0$
  such that for all bounded $f$ satisfying $\int f d\mu=0$,
  \begin{equation}\label{eq:decay}
    \NRM{P_t f}_{\dL^2(\mu)} \leq K_2\,\NRM{f}_\infty e^{-t}.
  \end{equation}
  We shall give a proof of the Lemma for $k=2$, the general case $k\geq2$
  being heavier but entirely similar. We set $s=t_1< t_2=t$. It is enough to
  show that for every bounded continuous functions $F,G:S^1\times\dR\to\dR$,
  we have the convergence
  $$
  \lim_{\varepsilon\to0}\dE[F(y_s^\varepsilon)G(y_t^\varepsilon)] %
  = \dE[F(Y)]\dE[G(Y)]
  $$
  where $Y$ is a random variable of law $\mu$. Since $y_0$ follows the law
  $\mu$, we can safely assume that the functions $F$ and $G$ have zero
  $\mu$-mean, and reduce the problem to show that
  $$
  \dE[F(y_s^\varepsilon)G(y_t^\varepsilon)] %
  =\int\!P_{s/\varepsilon^2}(FP_{(t-s)/\varepsilon^2}G)\,d\mu %
  =\int\!FP_{(t-s)/\varepsilon^2}G\,d\mu %
  \underset{\varepsilon\to0}{\longrightarrow}0.
  $$
  Now since $\mu$ is a probability measure, we have
  $\mathbb{L}^2(\mu)\subset\mathbb{L}^1(\mu)$ and thus
  $$
  \ABS{\int\!FP_{(t-s)/\varepsilon^2}G\,d\mu} %
  \leq \NRM{FP_{(t-s)/\varepsilon^2}G}_1 %
  \leq \NRM{FP_{(t-s)/\varepsilon^2}G}_2 %
  \leq \NRM{F}_\infty\NRM{P_{(t-s)/\varepsilon^2}G}_2.
  $$
  The desired result follows then from the $\mathbb{L}^2-\mathbb{L}^\infty$
  bound \eqref{eq:decay} since
  $$
  \NRM{P_{(t-s)/\varepsilon^2}G}_{\dL^2(\mu)} %
  \leq K_2\NRM{G}_\infty e^{-(t-s)/\varepsilon^2} %
  \underset{\varepsilon\to0}{\longrightarrow}0.
  $$    
\end{proof}

\begin{proof}[Proof of Theorem \ref{th:main}]
  The strategy is the usual one based on It\^o's formula, Poisson equation,
  and a martingale CLT. However, each step involves some peculiar properties
  of the stochastic process. For convenience we split the proof into small
  parts with titles.

  \subsubsection*{Poisson equation} 

  Let $L$, $L^*$, and $(P_t)_{t\geq0}$ be as in the proof of Lemma
  \ref{le:chaos}. Since $f$ is bounded and satisfies $\int f d\mu = 0$ (i.e.
  $f$ has zero $\mu$-mean), the bound \eqref{eq:decay} ensures that
  $$
  g = - \int_0^{\infty}\!P_s f\,ds \quad\in \dL^2(\mu).
  $$
  Furthermore, the formula $P_t g-g=\int_0^t\,P_sf\,ds$ ensures that
  $$
  \lim_{t\to0}\frac{1}{t}(P_t g-g)=f\quad\text{strongly in $\dL^2(\mu)$.}
  $$
  It follows that $g$ belongs to the $\dL^2(\mu)$-domain of $L$ and satisfies
  to the Poisson equation:
  $$
  Lg=f\quad\text{in $\dL^2(\mu)$}.
  $$
  Since $\mu$ has an everywhere positive density with respect to the Lebesgue
  measure on $S^1 \times \dR$, we immediately deduce that $g$ belongs to the
  set of Schwartz distributions $\mathcal D'$ and satisfies $Lg=f$ in this
  set. Since $L$ is hypo-elliptic (it satisfies the H\"{o}rmander brackets
  condition) and $f$ is $C^\infty$, it follows that $g$ belongs to $C^\infty$.
  Hence we have solved the Poisson equation $Lg=f$ in a strong sense. Remark
  that since $g\in \dL^2(\mu)$ and $f$ is bounded, we get
  $$
  \dE_{\mu}[2 \alpha^2 \, |\partial_\kappa \, g|^2] %
  = - \dE_{\mu}[g Lg] %
  = - \dE_{\mu}[g f] <\infty.
  $$ 

  \subsubsection*{It\^o's formula} 

  Since $g$ is smooth, we may use It\^o's formula to get
  $$
  g(y_t) - g(y_0) = \int_0^t\!\alpha \sqrt 2\,\partial_\kappa g(y_s)\,dB_s %
  + \int_0^t\!Lg(y_s)\,ds %
  \quad \text{almost surely}
  $$
  which can be rewritten thanks to the Poisson equation $Lg=f$ as
  \begin{equation}\label{eq:ito}
    \int_0^t\!f(y_s) \, ds %
    = g(y_t) - g(y_0) %
    - \alpha \sqrt 2 \int_0^t\!\partial_\kappa g(y_s)\,dB_s %
    \quad \text{almost surely}.
  \end{equation}
  This last equation \eqref{eq:ito} reduces the CLT for the process
  $$
  \PAR{\varepsilon\int_0^{t/\varepsilon^2}\!f(y_s) \, ds}_{t\geq0}
  $$
  to showing that $(\varepsilon(g(y_{t/\varepsilon^2})-g(y_0)))_{t\geq0}$ goes
  to zero as $\varepsilon\to0$ and to a CLT for the process
  $$
  \PAR{\alpha \varepsilon\sqrt 2
    \int_0^{t/\varepsilon^2}\!\partial_\kappa g(y_s)\,dB_s}_{t\geq0}.
  $$
  For such, we shall use the initial conditions and a martingale argument
  respectively.

  \subsubsection*{Initial condition}

  Since the law $\mu$ of $y_0$ is stationary, Markov's inequality gives for
  any constant $K>0$,
  $$
  \dP(|g(y_{t/\varepsilon^2})|\geq K/\varepsilon) %
  = \dP(|g(y_{0})|\geq K/\varepsilon) %
  \leq \frac{\Var_{\mu}(g) \, \varepsilon^2}{K^2}
  \underset{\varepsilon\to0}{\longrightarrow}0.
  $$
  It follows that any n-uple of increments
  $$
  \varepsilon \, (g(y_{t_1})-g(y_{t_0}),\ldots,g(y_{t_n})-g(y_{t_{n-1}}))
  $$
  converges to $0$ in probability as $\varepsilon\to0$. Thanks to
  \eqref{eq:ito}, this reduces the CLT for
  $$
  \PAR{\varepsilon\int_0^{t/\varepsilon^2}\!f(y_s) \, ds}_{t\geq0}
  $$
  to the CLT for
  $$
  (M_t^\varepsilon)_{t\geq0}:=
  \PAR{\varepsilon \alpha\sqrt 2
    \int_0^{t/\varepsilon^2}\!\partial_\kappa g(y_s)\,dB_s}_{t\geq0}.
  $$

  \subsubsection*{Martingale argument} 

  It turns out that $((M_t^\varepsilon)_{t\geq0})_{\varepsilon>0}$ is a
  family of local martingales. These local martingales are actually $\dL^2$
  martingales whose brackets (increasing processes)
  $$
  \DP{M^\varepsilon}_t %
  = \varepsilon^2 2 \alpha^2 %
  \int_0^{t/\varepsilon^2}\!|\partial_\kappa g|^2(y_s)\,ds
  $$ 
  converge almost surely to
  $$
  2\alpha^2 t \dE_{\mu}[|\partial_\kappa \, g|^2]= t \, V_f
  \quad\text{as $\varepsilon \to 0$}
  $$
  thanks to the Ergodic Theorem \eqref{eq:erg}. According to the CLT for
  $\dL^2$-martingales due to Rebolledo, see for example \cite{helland} for an
  elementary proof, it follows that the family $(M_t^\varepsilon)_{t\geq0}$
  converges weakly (for the Skorohod topology) to $V_f \, (B_t^\tau)_{t\geq0}$
  where $(B_t^\tau)_{t\geq0}$ is a standard Brownian Motion. Consequently, we
  obtain the desired CLT for the process
  $$
  \PAR{\varepsilon \int_0^{t/\varepsilon^2}\!\!\!\!f(y_s)\,ds}_{t\geq0}.
  $$
  Namely, its increments are converging in distribution as $\varepsilon\to0$
  to the law of a Brownian Motion with variance $V_\tau$. It remains to obtain
  the desired CLT for the process ${(z_t^\varepsilon)}_{t\geq0}$.
  
  \subsubsection*{Coupling with propagation of chaos and asymptotic independence} 

  By the result above and Lemma \ref{le:chaos}, the CLT for
  ${(z_t^\varepsilon)}_{t\geq0}$ will follow if we show that
  $$
  \PAR{\varepsilon\int_0^{t/\varepsilon^2}\!\!\!\!f(y_s)\,ds}_{t\geq0} %
  \quad\text{and}\quad %
  (y_{t/\varepsilon^2})_{t\geq0} %
  ={(\ote_{t/\varepsilon^2},\kappa_{t/\varepsilon^2})}_{t\geq0}
  $$
  are independent processes as $\varepsilon\to0$. It suffices to establish the
  independence as $\varepsilon\to0$ for an arbitrary $k$-uple of times $0=t_0
  < t_1 <\cdots< t_k=t$. To this end, let us introduce a bounded continuous
  function $h$ and the smooth bounded functions
  $$
  h_j(u)=e^{\sqrt{-1}b_j u}
  $$
  where $1\leq j\leq k$ for given real numbers $b_1,\ldots,b_k$. Let us define
  $$
  A^{\varepsilon} %
  = \dE_{\mu}\SBRA{ %
    h(y_{t/\varepsilon^2}) %
      \prod_{j=1}^k h_j\PAR{\varepsilon %
        \int_{t_{j-1}/\varepsilon^2}^{t_{j}/\varepsilon^2} \, f(y_s) \, ds}}
  $$
  Introduce $t_\varepsilon =(t/\varepsilon^2) - (t/\sqrt{\varepsilon})$ and
  $s_\varepsilon = t/\sqrt{\varepsilon}$. For $\varepsilon$ small enough,
  $t_\varepsilon > (t_{k-1}\varepsilon^2)$, so that using the Markov property
  at time $t_\varepsilon$ we get
  $$
  A^{\varepsilon} %
  = \dE_{\mu}\SBRA{\prod_{j=1}^{k-1}h_j %
    \PAR{\varepsilon\int_{t_{j-1}/\varepsilon^2}^{t_{j}/\varepsilon^2}%
      \!f(y_s)\,ds} %
     \dE_{\mu}\SBRA{h(y_{t/\varepsilon^2}) h_k\PAR{\varepsilon\int_{t_{k-1}/\varepsilon^2}^{t_{k}/\varepsilon^2}\!f(y_s)\,ds}\,\biggr\vert\,\cF_{t_\varepsilon}}} 
  $$
  The conditional expectation in the right hand side is equal to
  $$
  h_k %
  \PAR{\varepsilon\int_{t_{k-1}/\varepsilon^2}^{t_\varepsilon}\!f(y_s)\,ds}%
  \dE\SBRA{h(y_{s_\varepsilon})h_k\PAR{\varepsilon\int_{0}^{s_\varepsilon}\!f(y_s)\,ds\,\biggr\vert\,y_0=y_{t_\varepsilon}}}
  $$
  and the second term can be replaced by
  $$
  \dE\SBRA{h(y_{s_\varepsilon})\,\biggr\vert\,y_0=y_{t_\varepsilon}}
  $$
  up to an error less than 
  $$
  \varepsilon \NRM{h}_\infty \NRM{f}_\infty s_\varepsilon
  $$
  going to $0$ as $\varepsilon \to 0$. It thus remains to study 
  $$
  \dE_{\mu}\SBRA{\prod_{j=1}^{k-1}h_j %
    \PAR{\varepsilon\int_{t_{j-1}/\varepsilon^2}^{t_{j}/\varepsilon^2} %
      \!f(y_s)\,ds} %
    h_k\PAR{\varepsilon\int_{t_{k-1}/\varepsilon^2}^{t_\varepsilon}\!f(y_s)\,ds}%
    \dE\SBRA{h(y_{s_\varepsilon})\,\biggr\vert\,y_0=y_{t_\varepsilon}}}.
  $$
  Conditioning by $y_{t_\varepsilon}$, this can be written in the form 
  $$
  \int\,H(\varepsilon,y)\,P_{s_\varepsilon} h(y)\,\mu(dy)
  $$ 
  with a bounded $H$, so that using the convergence of the semi-group, we may
  again replace $P_{s_\varepsilon} h$ by $\int h d\mu$ up to an error term
  going to $0$. It remains to apply the previously obtained CLT in order to
  conclude to the convergence and asymptotic independence.
\end{proof}

\begin{remark}[More general models]\label{rm:gene}
  The proof of Theorem \ref{th:main} immediately extends to more general
  cases. The main point is to prove that $g$ solves the Poisson equation in
  $\dL^2(\mu)$. In particular it is enough to have an estimate of the form
  $$
  \NRM{P_t f}_{\dL^2(\mu)} \leq \alpha(t)\,\NRM{f}_\infty
  $$
  for every $t\geq0$ with a function $\alpha$ satisfying
  $$
  \int_0^{\infty}\!\alpha(s)\,ds <\infty.
  $$
  According to \cite{BCG}, a sufficient condition for this to hold is to find
  a smooth increasing positive concave function $\varphi$ such that the
  function $\alpha$ defined by
  $$
  \alpha(t) = \frac{1}{(\varphi\circ G_\varphi^{-1})(t)} %
  \quad\text{where}\quad %
  G_\varphi(u) =\int_1^u\!\frac{1}{\varphi(s)}\,ds
  $$
  satisfies the integrability condition above, and a Lyapunov function $H\geq
  1$ such that
  $$
  \int\!H\,d\mu < \infty %
  \quad\text{and}\quad %
  L^* H \leq -\varphi(H) + O(\mathds{1}_C)
  $$ 
  for some compact subset $C$. In particular we may replace the
  Ornstein-Uhlenbeck dynamics for $\kappa$ by a more general Kolmogorov
  diffusion dynamics of the form
  $$
  d\kappa_t = - \nabla V(\kappa_t) dt + \sqrt 2 \, dB_t.
  $$
  The invariant measure of $(\kappa_t,\theta_t)_{t\geq0}$ is then
  $e^{-V(\kappa)}d\kappa d\theta$. We refer for instance to \cite{DFG,BCG} for
  the construction of Lyapunov functions in this very general situation. For
  example, in one dimension, one can take $V'(x)=|x|^p$ for large $|x|$ and
  $0<p\leq 1$. Choosing $H(y)=|\kappa|^q$ for large $\kappa$ furnishes a
  polynomial decay of any order by taking $q$ as large as necessary. Actually,
  in this last situation, the decay rate is sub--exponential, see for example
  \cite{DFG,BCG}.
\end{remark}

\begin{remark}[Asymptotic covariance] It is worth noticing that if the
  asymptotic variance
  $$
  (AV)_f = %
  \lim_{t \to\infty} \, %
  \frac{1}{t} \, \dE_{\mu}\SBRA{\PAR{\int_0^t\!f(y_s)\,ds}^2}
  $$
  exists, then $V_f = (AV)_f$. Similarly we may consider complex valued
  functions $f$ and replace the asymptotic variance by the asymptotic
  covariance matrix which takes into account the variances and the covariance
  of the real and imaginary parts of $f$.
\end{remark}

\begin{proof}[Proof of Corollary \ref{co:dm}]
  We may now apply the previous theorem and the previous remark to the
  $2$-dimensional smooth and $\mu$-centered function $\tau$. The only thing
  we have to do is to compute the asymptotic covariance matrix. To this end,
  first remark that elementary Gaussian computations furnishes the following
  explicit expressions
  \begin{align}
    \kappa_t & = e^{-t} \, \kappa_0 + \sqrt{2} \, \alpha \, \int_0^t \,
    e^{s-t} \, dB_s \, , \label{eq:explicit}\\
    \theta_t & = \theta_0 + (1-e^{-t}) \, \kappa_0 + \sqrt{2} \, \alpha \, \int_0^t \, (1-e^{s-t})
    \, dB_s \, . \label{eq:explicitbis}
  \end{align}
  Since $x_t^1=\int_0^t \, \cos \theta_s \, ds$ and $x_t^2=\int_0^t \, \sin
  \theta_s \, ds$, Markov's property and stationarity yield
  \begin{align}\label{eqcov}
    \dE_{\mu}[x_t^1 \, x_t^2] & =\dE_{\mu}\SBRA{%
      \int_0^t\!(x_s^1\sin\theta_s + x_s^2\cos\theta_s)\,ds} \\
    & =
    \dE_{\mu}\SBRA{\int_0^t\!\int_0^s\!(\cos\theta_u\sin\theta_s+\sin\theta_u\cos
      \theta_s)\,du\,ds} \nonumber \\
    & =
    \int_0^t\!\int_0^s\!\dE_{\mu}[\sin(\theta_u+\theta_s)]\,du\,ds\nonumber\\%
    & = \int_0^t\!\int_0^s\!\dE_{\mu}\SBRA{\sin\PAR{2\theta_0 + 2(1-e^{u-s})
        \kappa_0 +
        \sqrt{2}\alpha\int_0^{s-u}\!(1-e^{v-(s-u)})\,dB'_v}}\,du\,ds\nonumber
  \end{align}
  where $(B'_t)_{t\geq0}$ is a Brownian Motion independent of
  $(\kappa_0,\ote_0)$. Since $\kappa_0$ and $\ote_0$ are also
  independent (recall that $\mu$ is a product law), we may first integrate
  with respect to $\ote_0$ (fixing the other variables), i.e. we have to
  calculate $\mathbb E_\mu(\sin(2\ote_0 + C))$ which is equal to $0$ since
  the law $\mu$ of $\ote_0$ is uniform on $[0,2\pi[$. Hence the
  $\mu$-covariance of $(x_t^1,x_t^2)$ is equal to $0$ (since this is a
  Gaussian process, both variables are actually independent), and similar
  computations show that the asymptotic covariance matrix is thus $\dD I_2$
  where
  $$
  \dD = \int_0^{\infty}\!e^{- \, \alpha^2 (s-1+e^{-s})}\,ds.
  $$
\end{proof}

\begin{proof}[Proof of Theorem \ref{th:out}]
  We assume now that $y_0\sim\nu$ instead of $y_0\sim\mu$. We may mimic the
  proof of Theorem \ref{th:main}, provided we are able to control
  $\dE_\nu(g^2(y_s))$. Indeed the invariance principle for the local
  martingales $(M_t^\varepsilon)_{t\geq 0}$ is still true for the
  finite-dimensional convergence in law, according for instance to \cite[Th.
  3.6 p. 470]{JS}. The Ergodic Theorem ensures the convergence of the
  brackets. The first remark is that these controls are required only for
  $s\geq s_0 \geq 0$ where $s_0$ is fixed but arbitrary. Indeed since $\tau$
  is bounded, the quantity
  $$
  \varepsilon\int_0^{s_0}\!\tau(\theta_s)\,ds 
  $$
  goes to $0$ almost surely, so that we only have to deal with
  $\int_{s_0}^{t/\varepsilon^2}$ so that we may replace $0$ by $s_0$ in all
  the previous derivation. Thanks to the Markov property we thus have to
  control $\dE_{\nu_{s_0}} (g^2(y_s))$ for all $s>0$, where $\nu_{s_0}$ denote
  the law of $y_{s_0}$. This remark allows us to reduce the problem to initial
  laws which are absolutely continuous with respect to $\mu$. Indeed thanks
  to the hypo-ellipticity of $\frac{\partial}{\partial t} + L$ we know that
  for each $s_0 > 0$, $\nu_{s_0}$ is absolutely continuous with respect to
  $\mu$.
  % (we know much more of course, see below, but again we exhibit a quite
  % general argument).
  Hence we have to control terms of the form
  $$
  \dE_{\mu}\SBRA{\frac{d\nu_{s_0}}{d\mu}(y_0) \, g^2(y_s)}.
  $$
  The next remark is that \cite[Theorem 2.1]{BCG} immediately extends to the
  $\dL^p$ framework for $2\leq p <\infty$, i.e. there exists a constant $K_p$
  such that for all bounded $f$ satisfying $\int f d\mu=0$,
  \begin{equation}\label{eq:decay_p}
    \NRM{P_t f}_{\dL^p(\mu)} \,  \leq K_p \, \NRM{f}_\infty \, e^{-t}.
  \end{equation}
  Since the function $f$ is bounded and satisfies $\int f d\mu = 0$, the
  previous bound ensures that $g$ belongs to $\dL^p(\mu)$, for all $p<\infty$.
  In particular, as soon as $d\nu_{s_0}/d\mu$ belongs to $\dL^q(\mu)$ for some
  $1<q$, $g(y_s)$ belongs to $\dL^2(\dP_\nu)$ for all $s\geq s_0$.
  Additionally, these bounds allow to show without much efforts that the
  ``propagation of chaos'' of Lemma \ref{le:chaos} still holds when the
  initial law is such a $\nu$. To conclude we thus only have to find
  sufficient condition for $d\nu_{s_0}/d\mu$ to belong to one $\dL^q(\mu)$
  ($q>1$) for some $s_0\geq 0$. Of course, a first situation is when this
  holds for $s_0=0$. But there are many other situations.

  Indeed recall that for non-degenerate Gaussian laws $\eta_1$ and $\eta_2$
  the density $d\eta_2/d\eta_1$ is bounded as soon as the covariance matrix of
  $\eta_1$ dominates (in the sense of quadratic forms) the one of $\eta_2$ at
  infinity, i.e. the associated quadratic forms satisfy $q_1(y)>q_2(y)$ for
  $|y|$ large enough. According to \eqref{eq:explicit} and
  \eqref{eq:explicitbis} the joint law of $(\kappa_t,\theta_t)$ starting from a
  point $(\kappa,\theta)$ is a $2$-dimensional Gaussian law with
  mean $$m_t=(e^{-t} \kappa, \theta + (1-e^{-t}) \kappa)$$ and covariance
  matrix $D_t = \alpha^2 \, A_t$ with 
  $$
  A_t = %
  \begin{pmatrix}
    1-e^{-2t} & (1-e^{-t})^2 \\
    (1-e^{-t})^2 & 2t - 3 + 4e^{-t} - e^{-2t}
  \end{pmatrix}.
  $$
  Note that if the asymptotic covariance of $(\kappa_t,\theta_t)$ is not $0$,
  the asymptotic correlation vanishes, explaining the asymptotic
  ``decorrelation'' of both variables. It is then not difficult to see that if
  $\nu=\delta_y$ is a Dirac mass, then $d\nu_{s}/d\mu$ is bounded for every
  $s>0$. Indeed for $t$ small enough, $A_t$ is close to the null matrix, hence
  dominated by the identity matrix. It follows that $d\nu_t/d\eta$ is bounded,
  where $\eta$ is a Gaussian variable with covariance matrix $\alpha^2 I_2$.
  The result follows by taking the projection of $\theta$ onto the unit
  circle. A simple continuity argument shows that the same hold if $\nu$ is a
  compactly supported measure.  
\end{proof}

\begin{remark}
  Once obtained such a convergence theorem we may ask about explicit bounds
  (concentration bounds) in the spirit of \cite{CatGui2} (some bounds are
  actually contained in this paper). One can also ask about Edgeworth
  expansions etc. However, our aim was just to give an idea of the stochastic
  methods than can be used for models like the PTWM.
\end{remark}

\begin{remark}
  The most difficult point was to obtain $\dL^p(\mu)$ estimates for
  $\partial_\kappa g$. Specialists of hypo-elliptic partial differential
  equations will certainly obtain the result by proving quantitative versions
  of H\"{o}rmander's estimates (holding on compact subsets $U$):
  $$
  \NRM{\partial_\kappa g}_p \leq C(U) \, (\NRM{g}_p + \NRM{Lg}_p).
  $$  
\end{remark}

We end up the present paper by mentioning an interesting and probably
difficult direction of research, which consists in the study of the long time
behavior of interacting copies of PTWM--like processes, leading to some kind
of kinetic hypo-elliptic mean-field/exchangeable Mac Kean-Vlasov equations
(see for example \cite{MR1108185,MR1431299} and references therein for some
aspects). At the Biological level, the study of the collective behavior at
equilibrium of a group of interacting individuals is particularly interesting,
see for instance \cite{vicsek}.

{\footnotesize %

\bibliographystyle{plain}
\bibliography{poissons}

\bigskip

\vfil

\bigskip

\noindent
Patrick \textsc{Cattiaux},
E-mail: \texttt{cattiaux(AT)math.univ-toulouse.fr}\\
S\'ebastien
\textsc{Motsch},
E-mail: \texttt{sebastien.motsch(AT)math.univ-toulouse.fr}

\smallskip

\noindent
\textsc{UMR5219 CNRS \& Institut de Math\'ematiques de Toulouse} \\
\textsc{Universit\'e de Toulouse}\\
\textsc{118 route de Narbonne, F-31062 Toulouse Cedex 09, France.}\\

\noindent
Djalil~\textsc{Chafa{\"{\i}}},
E-mail: \texttt{chafai(AT)math.univ-toulouse.fr} 

\smallskip
\noindent
\textsc{UMR181 INRA ENVT \& Institut de Math\'ematiques de Toulouse} \\
\textsc{\'Ecole Nationale V\'et\'erinaire de Toulouse, Universit\'e de
  Toulouse.}\\
\textsc{23 chemin des Capelles, F-31076 Toulouse, Cedex 3, France.}

\vfill

\begin{flushright}
  \texttt{Compiled \today}
\end{flushright}

} %footnotesize

\end{document}